\documentclass[11pt]{article}

\usepackage[utf8]{inputenc}
\usepackage[T1]{fontenc}
\usepackage[english,french]{babel}
\usepackage{fancyhdr}		
\usepackage{graphicx}
\usepackage{colortbl}
\usepackage{footnote}
\usepackage{makeidx}
\usepackage{amsfonts}
\usepackage{mathrsfs}
\usepackage{indentfirst}
\usepackage{shorttoc}
\usepackage{amsmath, amsthm, amssymb}
\usepackage{geometry}
\usepackage{multicol}
\usepackage{multirow}
\usepackage[all]{xy}
\usepackage{young}
\usepackage{stmaryrd}
\usepackage{dsfont}
\usepackage{tikz}
\usepackage{MnSymbol}
\usetikzlibrary{shapes}
\usepackage{tikz-cd}
\usetikzlibrary{arrows}

\def\ker{\mathop{\rm ker}\nolimits}
\def\Im{\mathop{\rm Im}\nolimits}

\newtheoremstyle{plain}
{\topsep}
{\topsep}
{\itshape}
{}
{\bfseries}
{.}
{ }
{\thmname{#1}\thmnumber{ \textup{#2}}\thmnote{~: \textup{#3}}}
\theoremstyle{plain}
\newtheorem{theoreme}{Theorem}[section]
\newtheorem{corollary}[theoreme]{Corollary}
\newtheorem{proposition}[theoreme]{Proposition}
\newtheorem{lemma}[theoreme]{Lemma}
\theoremstyle{definition}

\theoremstyle{definition}
\newtheorem*{remark}{Remark}
\theoremstyle{definition}
\newtheorem*{remarks}{Remarks}
\theoremstyle{definition}
\newtheorem*{example}{Example}
\theoremstyle{definition}

\usepackage{etoolbox}

\usepackage{lineno,hyperref}

\begin{document}

\title{Unstable points for torus actions on flag varieties}
\author{\textsc{Benoît Dejoncheere}\footnote{\texttt{dejoncheere@math.univ-lyon1.fr}}}
\date{}

\selectlanguage{english}

\maketitle

\textsc{Abstract :} {\small In this paper, we will look at actions on complex flag varieties $G/P$ of the torus $\hat{T}=\bigcap\limits_{\alpha\in\Delta\setminus\Delta_P}\ker(\alpha)$, and under reasonable assumptions, we will give a description of the set $X^{us}$ of unstable points for $\hat{T}$-linearized invertible sheaves. We will investigate the case where $P$ is a maximal parabolic subgroup, and show that $X^{us}$ can be written as a disjoint union of a Schubert variety and an opposite Schubert variety, and we deduce the vanishing of cohomology groups $H^i(Y,{\cal M})$ for invertible sheaves ${\cal M}$ on the quotient variety $Y$ for $i$ in a range given by the codimension of $X^{us}$.}

\tableofcontents

\section{Introduction}
If $G$ is a reductive complex algebraic group acting on a variety $X$, the quotient space $X/G$ cannot in general be equipped with a structure of algebraic variety. The framework of Geometric Invariant Theory is useful to construct rational maps $X\to Y$ and open subsets $U\subset X$ such that $U\to Y$ is a quotient map with good properties. We will recall in section \ref{sec2} known facts about GIT, and we will translate what is needed in our setting. In this paper, we will look at torus actions on flag varieties via GIT. More precisely, we will obtain under reasonable assumptions a description of the unstable locus on flag variety $G/P$ for ample invertible sheaves linearized by a torus canonically associated to $P$ with Theorem \ref{unstableflags}, and this description will be more accurate in the case of a maximal parabolic subgroup thanks to Proposition \ref{uspointsmaxpara}.\\
\indent The first motivation to this work is given by the work of H. Seppänen and V. Tsanov in \cite{ST}, which consists of the investigation of semisimple group $\hat{G}$ actions on full flag varieties $G/B$, and torus actions correspond to the case "dual" to the semisimple case. The main difference is that when looking at $\hat{G}$-actions on flag varieties, there much less possible linearizations, but more one-parameter subgroups than in the case of a torus action. The second motivation is given by the study of wonderful compactifications of some symmetric spaces done in \cite{De}. It appears that all these varieties can be described as GIT quotients by $\mathbb{C}^*$ of flag variety $G/P$ with $P$ maximal parabolic subgroup of $G$. This examples will be recalled in the section \ref{sec4-1}.\\
\indent Let $G$ be a complex semisimple algebraic group, $B$ be a Borel subgroup of $G$ and $T$ a maximal torus in $B$. In section \ref{sec3-1}, we will give an expression of the set of unstable points $X^{us}$ relatively to a $T$-linearized invertible sheaf on $X=G/B$ as a union of Schubert cells, and we will generalize this under nice conditions to $X=G/P$. Namely, to mimic the nice combinatorics of the Weyl group $W$ on the character lattice ${\cal X}$ of $T$ and to avoid fixed points issues, we will have to look at actions of the subtorus
\[
\hat{T}:=\bigcap\limits_{\alpha\in\Delta\setminus\Delta_P}\ker(\alpha)
\]
on $X$, and we will define a subgroup $\hat{W}\subset W^P$ which will have to play the same role than $W$ for the character lattice $\hat{\cal X}$ of $\hat{T}$.\\
\indent In the last section, we will look closer to the case where $P$ is a maximal parabolic subgroup of $G$, and we will refine our description to show that $X^{us}$ is actually a disjoint union of a Schubert variety $\overline{X_w}$ and of an opposite Schubert variety $w_{0,P}\overline{X_w'}$. We will also give an explicit description of $w$ and $w'$ that could occur in such a way. This description allows us to compute the codimension of $X^{us}$, that can be done when knowing only about $W^P$. Since $X$ is a flag variety, one obtain when the codimension $l$ of $X^{us}$ is at least two the cancellation of the cohomology groups $H^i(Y,{\cal M})$ for ${\cal M}$ invertible sheaf on the quotient $Y$ and for $1\leq i\leq l-1$.

\section{Generalities about GIT quotients}\label{sec2}
We will recall in this section some standard facts about GIT quotients and about their variations, following \cite{MFK}, \cite{DH} and \cite{BP}. Let $X$ be a complex irreducible algebraic variety acted on by a complex reductive group $G$. We say that $\pi : X\to Y$ is a categorical quotient if it is $G$-invariant and if any $G$-invariant morphism $X\to Z$ factors through $\pi$. We say that a categorical quotient is good if it sends disjoints $G$-invariant closed subsets of $X$ to disjoints closed subsets of $Y$, and if ${\cal O}_Y(U)={\cal O}_X(\pi^{-1}(U))^G$ for all open subset $U\subset Y$. We say that a good categorical quotient is a geometric quotient if moreover $\pi$ is surjective, and if its fibres are $G$-orbits.\\
\indent
Denote by $\textrm{Pic}^G(X)$ the set of $G$-linearized invertible sheaves on $X$ modulo isomorphism. Let ${\cal L}$ be such a sheaf. A point $x\in X$ is called semi-stable (with respect to ${\cal L}$) if there is an integer $n>0$ and an invariant section $s\in H^0(X,{\cal L}^{\otimes n})^G$ such that its non-zero locus $X_s$ is affine and $x\in X_s$. We will say that $x$ is stable if it is semi-stable, if we can choose such a $s$ such that $G$-orbits in $X_s$ are closed, and if its stabilizer $G_x$ is finite. A point $x\in X$ is said unstable if it is not semi-stable. The semi-stable (resp. stable, unstable) locus with respect to ${\cal L}$ is denoted by $X^{ss}({\cal L})$, (resp. $X^s({\cal L})$, $X^{us}({\cal L})$). There exists a good categorical quotient $\pi : X^{ss}({\cal L})\to Y({\cal L})$ whose restriction from $X^{s}({\cal L})$ to its image is a geometric quotient.\\
\indent
We will now focus in the case where $X$ is a projective variety, and ${\cal L}$ is ample. It follows that the algebra $\bigoplus\limits_{n\geq 0}H^0(X,{\cal L}^{\otimes n})$ is finitely generated, and $G$ being reductive, so is its algebra of $G$-invariants. We have an isomorphism
\[
Y({\cal L})=\textrm{Proj}(\bigoplus\limits_{n\geq 0}H^0(X,{\cal L}^{\otimes n})^G)
\]
We also have a good characterization of (semi-)stability thanks to the Hilbert-Mumford criterion, that we are going to recall. For all $x\in X$ and for all one-parameter subgroup $\lambda : \mathbb{C}^*\to G$, let $x_0=\lim\limits_{t\to 0}\lambda(t).x$, and define $\mu^{\cal L}(x,\lambda)$ as the integer such that $\lambda(\mathbb{C}^*)$ acts on the fibre ${\cal L}_{x_0}$ with the character $t\mapsto t^{\mu^{\cal L}(x,{\lambda})}$. The Hilbert-Mumford criterion can then be stated as :
\begin{proposition}[Hilbert-Mumford Criterion]
Let ${\cal L}$ be a $G$-linearized ample invertible sheaf on $X$, and $x\in X$. Then
\begin{enumerate}
\item[(1)] $x\in X^{ss}({\cal L})\Leftrightarrow \forall \lambda : \mathbb{C}^*\to G, \mu^{\cal L}(x,\lambda)\leq 0$ ;
\item[(2)] $x\in X^{s}({\cal L})\Leftrightarrow \forall \lambda : \mathbb{C}^*\to G, \mu^{\cal L}(x,\lambda)< 0$.
\end{enumerate}
\end{proposition}
A natural question arising is the following : how to compare different quotients with respect to two different $G$-linearized ample invertible sheaves ${\cal L}$ and ${\cal L'}$ ? These relations are encoded by the so-called $G$-ample cone, of which we are going to recall the definition. We will say that ${\cal L}\in \textrm{Pic}^G(X)$ is $G$-homologically trivial if its first Chern class is trivial, and if it has a trivial $G$-linearization, and we will denote by $\textrm{Pic}^G(X)_0$ the subgroup of $G$-homologically trivial ${\cal L}\in \textrm{Pic}^G(X)$ (for such a ${\cal L}$, one can define a trivial $G$-action on the total space $L$ of ${\cal L}$ by $g.(x,f)=(g.x,f)$, \cite{DH} 2.3.3). Let $\textrm{NS}^G(X):=\textrm{Pic}^G(X)/\textrm{Pic}^G(X)_0$ be the $G$-Néron-Severi group.
\begin{remarks}
\item[1)] Moding out by $\textrm{Pic}^G(X)_0$ is not a problem when looking at GIT quotients : if ${\cal L}\in\textrm{Pic}^G(X)_0$ is ample, then for all $x\in X$ and for all one-parameter subgroup $\lambda : \mathbb{C}^*\to G$, $\mu^{\cal L}(x,\lambda)=0$, hence homologically equivalent $G$-linearized ample invertible sheaves share the same (semi-)stable points. Moreover, if a class in $\textrm{NS}^G(X)$ can be represented by an ample ${\cal L}$, then all invertible sheaves in this class are ample (this follows from Kleiman ampleness criterion, cf. \cite{La} 1.4.26).
\item[2)] We should mention that in the case we will be interested in later, namely for flag varieties, $\textrm{Pic}(X)_0$ (hence $\textrm{Pic}^G(X)_0$) will be trivial. This follows from  $H^1(X,{\cal O}_X)=0$, which says that the first Chern class map $\textrm{Pic}(X)\to H^2(X,\mathbb{Z})$ is injective.
\end{remarks}
For lighter denominations, we will simply say that ${\cal L}\in \textrm{NS}^G(X)$ is $G$-ample if it can be represented by a $G$-linearized ample invertible sheaf, and if it has semi-stable points. Now one should notice that if ${\cal L}$ and ${\cal M}$ are $G$-ample, then so is ${\cal L}\otimes{\cal M}$, hence $G$-ample invertible sheaves span a semigroup in $\textrm{NS}^G(X)$. The $G$-ample cone is then defined as the convex cone spanned by $G$-ample sheaves in $\textrm{NS}^G(X)_{\mathbb{R}}:=\textrm{NS}^G(X)\otimes \mathbb{R}$, and will be denoted by $C^G(X)$.
\begin{remark}
It still makes sense to speak about (semi-)stability for $x\in X$ wih respect to $l\in C^G(X)$. In \cite{DH}, the authors constructed from the functions $\mu^{\textbullet}(x,\lambda)$ a unique \textit{continuous} function $M^{\textbullet}(x)$ (which has to be thought as a weighted supremum of $\mu^{\textbullet}(x,\lambda)$, with weights depending on $\lambda$ only) such that
\begin{itemize}
\item[(1)] for ample ${\cal L}\in \textrm{NS}^G(X)$ and $x\in X$, $x\in X^{ss}({\cal L})\Leftrightarrow M^{\cal L}(x)\leq 0$ and $x\in X^{s}({\cal L})\Leftrightarrow M^{\cal L}(x)<0$ ;
\item[(2)] for $l,l'\in C^G(X)$ and $x\in X$, $M^{\lambda l}(x)=\lambda M^l(x)$ and $M^{l+l'}(x)\leq M^{l}(x)+M^{l'}(x)$ ;
\end{itemize}
and it is easy to extend Hilbert-Mumford criterion to the whole cone $C^G(X)$ thanks to these properties.
\end{remark}
We will say that $H\subset C^G(X)$ is a wall if there exists a $x\in X$ with $\dim G_x>0$ such that $H=\{l\in C^G(X)|M^l(x)=0\}$. A connected component of the complement of the union of all walls is called a chamber. Remark that $l$ and $l'$ are in the same chamber if and only if $X^{ss}(l)=X^{s}(l)=X^{s}(l')=X^{ss}(l')$. We can refine this partition of $C^G(X)$ by using the definition of cell of $C^G(X)$ ; a cell of $C^G(X)$ is a connected component of an equivalence class for the relation
\[
l\simeq l' \Leftrightarrow X^{ss}(l)\setminus X^{s}(l)=X^{ss}(l')\setminus X^{s}(l')
\]
Then one can notice that the (semi-)stable locus of $l\in C^G(X)$ is the same inside the cell $F$ containing $l$, and we will denote them by $X^{ss}(F)$ and $X^{s}(F)$. Before focusing on the case of quotients by tori, it is worth mentioning the following property, which is proved in \cite{DH}(4.1.5) for cells which are not chambers, but whose proof is exactly the same for general cells.
\begin{proposition}\label{inverselimit}
Let $F$ be a cell with a nonempty intersection with the closure of another cell $F'$. Then
\item[(1)] $X^{ss}(F')\subsetneq X^{ss}(F)$ and $X^{s}(F)\subset X^{s}(F')$ ;
\item[(2)] The inclusion $X^{ss}(F')\subsetneq X^{ss}(F)$ induces a morphism $f_{F',F}:Y(F')\to Y(F)$, which is surjective and birationnal when $X^{s}(F)$ is nonempty.
\end{proposition}

From now on, we will assume that $G$ is a torus $T\simeq (\mathbb{C}^*)^n$. Let $X$ be an irreducible projective $T$-variety, and let ${\cal L}\in\textrm{Pic}(X)$ be very ample. Then we have an embedding $i : X\hookrightarrow \mathbb{P}(V)$, where $V$ is a finite-dimensional $T$-module, such that ${\cal L}=i^*{\cal O}(1)$. Thanks to the $T$-action on $V$, we can decompose
\[
V=\bigoplus\limits_{\chi\in{\cal X}(T)}V_{\chi}
\]
as a sum of characteristic subspaces. If $x=[v]\in X$, we have $v=\sum v_{\chi}$, and we denote by $\Pi(x)$ the convex hull in ${\cal X}(T)_{\mathbb{R}}:={\cal X}(T)\otimes_{\mathbb{Z}}\mathbb{R}$ of $\{\chi\in{\cal X}(T)|v_{\chi}\neq 0\}$, and by $\mathring{\Pi}(x)$ its interior.\\
\indent
Let $L({\cal L})$ denote the semigroup of ${\cal L}'\in\textrm{NS}^T(X)$ whose image in $\textrm{NS}(X)$ by forgetting the $T$-linearization is the homological equivalence class of some positive power of ${\cal L}$, and let $L({\cal L})_{\mathbb{K}}$ be the convex cone in $\textrm{NS}^T(X)_{\mathbb{K}}$ spanned by $L({\cal L})$, where $\mathbb{K}$ is $\mathbb{Q}$ or $\mathbb{R}$ (which will be a ray in $\textrm{NS}^T(X)_{\mathbb{K}}$). Let $l\in L({\cal L})_{\mathbb{Q}}$. Then there are positive $q,m$ such that 
\[
l=\frac{1}{q}({\cal L}^{\otimes m},\sigma)
\]
where $\sigma$ is a $T$-linearization of ${\cal L}^{\otimes m}$. For all $x\in X$, $\sigma$ induces for all $t\in T$ a morphism
\[
\begin{array}{rcl}
\sigma_x(t):({\cal L}^{\otimes m})_{x}&\to &({\cal L}^{\otimes m})_{tx}\\
s&\mapsto &\chi_x(t)s(t.)
\end{array}
\]
for some $\chi_x\in {\cal X}(T)$. Now notice that $\chi:=\chi_x$ does not depend on $x$ (since $X$ is irreducible, it is enough to cover $X$ by $T$-invariants open affine subsets and use the rigidity of tori), and it does not depend on the ${\cal L}'$ in $\textrm{Pic}^T(X)$ representing $ql$ (since $T$-linearization on $\textrm{Pic}^T(X)_0$ are trivial). Hence if we denote by $RL({\cal L})_{\mathbb{K}}=L({\cal L})_{\mathbb{K}}/\mathbb{K}^*_+$ the set of rays of $L({\cal L})_{\mathbb{K}}$, we have an identification
\[
\begin{array}{rcl}
\Phi : RL({\cal L})_{\mathbb{Q}}&\to & {\cal X}(T)_{\mathbb{Q}}\\
\mathbb{Q}^*_+ l&\mapsto &\frac{1}{mq}\chi
\end{array}
\]
which can be extended by continuity to a map $RL({\cal L})_{\mathbb{R}}\to {\cal X}(T)_{\mathbb{R}}$. Remark that if $p=\frac{1}{m}\chi\in{\cal X}(T)_{\mathbb{Q}}$, $p$ corresponds to the action of $T$ through $\chi$ on $X\hookrightarrow \mathbb{P}(S^mV)$. This allows us to write $l\in C^T(X)$ as a couple $({\cal L},p)$ with ${\cal L}$ very ample in $\mathbb{R}^*_+l$ and $p\in {\cal X}(T)_{\mathbb{R}}$. Remark that $p$ does not depend on the choice of an embedding $i:X\hookrightarrow \mathbb{P}(V)$ such that ${\cal L}=i^*{\cal O}(1)$, and that this couple is not unique (since such a ${\cal L}$ is not).\\
\indent
We can now write again Hilbert-Mumford criterion :
\begin{proposition}
Let $l=({\cal L},p)\in C^T(X)$, and $x\in X$. Let $i:X\hookrightarrow \mathbb{P}(V)$ an embedding such that ${\cal L}=i^*{\cal O}(1)$. Then
\item[(1)] $x\in X^{ss}(l)$ iff $p\in \Pi(x)$ ;
\item[(2)] $x\in X^{s}(l)$ iff $p\in \mathring{\Pi}(x)$.
\end{proposition}
Again, these conditions do not depend on the choice of $i$.
\begin{proof}
Assume first that $l\in C^T(X)\cap \textrm{NS}^T(X)_{\mathbb{Q}}$, and let us write $l=({\cal L},\nu)$ with $\nu\in {\cal X}(T)$. $x=[v]\in X$ is semi-stable for $l$ if and only if for all one-parameter subgroup $\lambda : \mathbb{C}^*\to T$, $\lim\limits_{t\to 0}\lambda(t).x\neq 0$. By writing $v=\sum v_{\chi}$, one gets
\[
\lambda(t).x=\sum (\chi-\nu)(\lambda(t))v_{\chi}
\]
hence for all $\lambda\in {\cal X}_*(T)$ there exists $\chi\in {\cal X}(T)$ with $v_{\chi}\neq 0$ such that $\langle \chi,\lambda\rangle\leq\langle \nu,\lambda \rangle$, which exactly means for $\nu$ to be in the convex hull of $\Pi(x)$. This characterization extends by continuity to $C^T(X)$, and the same holds for stable points by replacing $\leq$ by $<$.
\end{proof}

\section{An expression of the set of unstable points}\label{sec3}

\subsection{Unstable points for full torus actions on full flag varieties}\label{sec3-1}
Let $G$ be a complex semi-simple connected algebraic group, let $T\subset B\subset G$ be respectively a maximal torus and a Borel subgroup of $G$. Let $\Phi=\Phi(G,T)$ be the root system associated to $T$, and $\Delta$ the set of simple roots relatively to $B$. Let ${\cal L}\in C^T(X)\cap\textrm{NS}^T(X)_{\mathbb{Q}}$ be a rational $T$-ample invertible sheaf on $X=G/B$, and let $l$ be a very ample representative of ${\cal L}^{\otimes n}$ for some $n>0$, which can be written as $l=({\cal L}_{\lambda},q)$ for some $\lambda$ and $q\in {\cal X}$ (where ${\cal X}$ denote the character lattice for $T$). Remark we can take $\lambda$ in ${\cal X}$ instead that in the weight lattice $\Lambda$ up to taking some positive power of $l$. Then ${\cal L}_{\lambda}=i^*{\cal O}(1)$, with $i:X\hookrightarrow \mathbb{P}(V_{\lambda})$. Let $W$ be the Weyl group of $T$.\\
\indent
We will call a one parameter subgroup $\xi\in{\cal X}^*$ regular if $X^T=X^{\xi}$. Since $X$ is a flag variety, the $T$-fixed points are exactly the $x_w=wB/B$ with $w\in W$, and for $\xi\in{\cal X}^*$, being regular is equivalent to say that $\langle\alpha,\xi\rangle\neq 0$ for all $\alpha\in\Delta$. Moreover, we will say that $x\in X$ is unstable relatively to $\xi\in{\cal X}^*$ if $\mu^{\cal L}(x,\xi)>0$. Then $x\in X$ is unstable iff $\exists \xi\in{\cal X}^*$ such that $x$ is unstable relatively to $\xi$. We define for $\hat{w}\in W$
\[
W(\hat{w},\lambda,q):=\{w\in W,\exists\xi\in {\cal X}^*\textrm{ with } \xi\textrm{ dominant regular },\langle w\lambda-\hat{w}q,\xi\rangle >0\}
\]
\indent Then we have the following description of unstable points for ${\cal L}$ :
\begin{proposition}\label{unstablefullflags}
With ${\cal L}$, $\lambda$ and $q$ as previously, and if we denote by $X_w$ the Bruhat cell $BwB/B$, then
\[
X^{us}({\cal L})=\bigcup\limits_{\hat{w}\in W}\bigcup\limits_{w\in W(\hat{w},\lambda,q)}\hat{w}^{-1}X_w
\]
\end{proposition}
\begin{remark}
Since $X^{us}$ is closed, this result is still true by replacing $X_w$ by their closures. However, writing the set of unstable points as a union of affine spaces could be helpful to compute the associated local cohomology groups.
\end{remark}
\begin{proof}
Roughly speaking, the idea of the proof is to use the Hilbert-Mumford criterion when the one-parameter subgroup first runs in a dominant chamber of ${\cal X}^*$, and then to translate by elements of $W$.\\
\indent Let $\xi\in{\cal X}^*$ be a dominant regular one parameter subgroup, and let $x=bwB/B\in X$ for some $b\in B$ and $w\in W$. Define
\[
x_0=\lim\limits_{t\to 0}\xi(t).x
\]
Since $\xi$ is regular, $x_0$ lies in $X^{T}$, and since it is dominant, seeing $X\subset \mathbb{P}(V_{\lambda})$, we get that $x_0=[v_{w\lambda}]$. Then we have
\[
\xi(t).v_{w\lambda}=(w\lambda-q)(\xi(t)).v_{w\lambda}
\]
and $x$ is unstable relatively to $\xi$ iff
\[
\langle w\lambda,\xi\rangle > \langle q,\xi\rangle
\]
\indent
Now let $\xi\in {\cal X}^*$ be regular, but not necessarily dominant. In that case, there exists a unique $\hat{w}\in W$ such that $\hat{w}\xi$ is dominant. Let $x\in X$, that will be written as $x=\hat{w}^{-1}bwB/B$ for some $b\in B$ and $w\in W$. Then $x_0=\hat{w}^{-1}wB/B$, and $x$ is unstable relatively to $\xi$ iff $\hat{w}x$ is unstable relatively to $\hat{w}\xi$, ie.
\[
\langle w\lambda, \hat{w}\xi\rangle > \langle \hat{w}q,\hat{w}\xi\rangle
\]
\indent
The last case if when $\xi\in\hat{\cal X}^*$ is not regular. There exists a (non-unique!) $\hat{w}\in W$ such that $\hat{w}\xi$ is dominant. Let $x\in X$. Then it lies in some translated Bruhat cell $\hat{w}^{-1}BwB/B$, for some $w\in W$, and $x_0$ lies in the same cell. Indeed, by writing $b=b_0b_+$ with 
\[
b_0\in\Im(\prod\limits_{\alpha >0, \langle\alpha,\hat{w}\xi\rangle=0}U_{\alpha}\to B) \textrm{ and } b_+\in\Im(\prod\limits_{\alpha >0, \langle\alpha,\hat{w}\xi\rangle>0}U_{\alpha}\times T\to B)
\]
we get
\[
(\hat{w}\xi)(t)bwB/B=b_0(\hat{w}\xi)(t)b_+(\hat{w}\xi)(t)^{-1}wP/P\underset{t\to 0}{\longrightarrow}b_0wB/B
\]
Then in $\mathbb{P}(V_{\lambda})$, we have
\[
\hat{w}x_0=[\sum\limits_{\chi\geq 0}\alpha_{\chi}v_{w\lambda+\chi}]
\]
for some constants $\alpha_{\chi}$ (with $\alpha_0\neq 0$). Then $x$ is unstable relatively to $\xi$ iff
\[
\forall \chi\geq 0, \langle w\lambda+\chi,\hat{w}\xi\rangle > \langle \hat{w}q,\hat{w}\xi\rangle
\]
which is, since $\chi$ has to be dominant, equivalent to the same condition for $\chi=0$. But this precisely means that the $T$-fixed point $\hat{w}^{-1}wB/B$ is also unstable relatively to $\xi$, hence $x$ is unstable relatively to $\xi$ iff the whole translated Bruhat cell is. It remains now to show that non-regular one-parameter subgroups do not involve extra unstable points.\\
\indent
Let $x=bwB/B\in X$ be unstable relatively to a non-regular dominant one-parameter subgroup $\xi\in {\cal X}^*$. Then by continuity, there exists a regular dominant $\xi'\in {\cal X}^*_{\mathbb{Q}}$ such that $\langle w\lambda-q,\xi'\rangle >0$, and $\xi'$ can be chosen in ${\cal X}^*$ (since $x$ will also be unstable relatively to $m\xi'$ for positive $m$), and the same can be done when translating by $\hat{w}\in W$. This completes the proof.
\end{proof}

\subsection{More general case}\label{sec3-2}

Let $G, B, T$ be as previously, let $P$ be a parabolic subgroup containing $B$. Set $X=G/P$. Define $\Delta_P\subset \Delta$ as the set of simple roots $\alpha$ such that the one-parameter subgroup $U_{\alpha}$ lies in the unipotent radical $U$ of $P$. To $P$ we will associate the subtorus
\[
\hat{T}:=\bigcap\limits_{\alpha\in\Delta\setminus\Delta_P}\ker(\alpha)
\]
of $T$, which can be seen as a "supplement" in $T$ of the maximal torus $T_P\subset T$ of the Levi subgroup $L_P$. Let $\hat{\cal X}:={\cal X}(\hat{T})$ and $\hat{\cal X}^*:={\cal X}^*(\hat{T})$. The inclusion $\hat{T}\hookrightarrow T$ induces an inclusion $j:\hat{\cal X}^*\to {\cal X}^*$. Let $W_P$ be the Weyl group of $P$ and $W^P:=W/W_P$. A coset $wW_P\in W^P$ will always be represented by $w\in W$ of minimal length. We would like to investigate on the unstable locus $X^{us}(l)$ for $\hat{T}$-very ample invertible sheaf $l=({\cal L}_{\lambda},q)$ on $X=G/P$ (with $\lambda$ in the weight lattice $\Lambda_P$ of $P$, and $q\in \hat{\cal X}$) with the same kind of ideas than before.\\
\indent Now let us look at the main tools of the previous proof to see if it adapts in our new setting. First of all, we still have "enough" regular one-parameter subgroups in $\hat{\cal X}^*$, and we do not have to much $\hat{T}$-fixed points in $X$, which are exactly the $wP/P$ with $w\in W^P$. Let us denote by $N$ the subset of nonregular one-parameter subgroups of $\hat{T}$. We will call a \textit{regular chamber} a connected component of $\hat{\cal X}^*_{\mathbb{R}}\setminus N_{\mathbb{R}}$. Then for a given $x\in X$, we have the same limit $x_0$ for $\xi$ in a fixed regular chamber. Indeed, if we write one-parameter subgroups as a linear combination of fundamental coweights, the previous claim is true since regularity coincides with coefficients being all nonzero.
\\
\indent The second problem is that in this new setting, we do not necessarily have a $\hat{w}\in W^P$ sending a regular one-parameter subgroup in $\hat{\cal X}^*$ to a dominant one (where in this setting, $\xi\in\hat{\cal X}^*$ dominant means $j(\xi)$ is). One simple example that does not satisfies this condition is given by $G=\textrm{SL}_3$ with standard $B$ and $T$, with $P$ the parabolic subgroup associated to $\alpha_1$, and $\hat{T}=\check{\varpi_1}$ : here $W^P=\{1,s_1,s_2s_1\}$, and $s_1\check{\varpi_1}$ and $s_2s_1\check{\varpi_1}$ are both different from $-k\check{\varpi_1}$ for $k>0$.  We will ask for existence of such a $\hat{w}$ for each one-parameter subgroup in $\hat{\cal X}^*$, and let
\[
\hat{W}=\{\hat{w}\in W^P, \exists \xi\in\hat{\cal X}^* \textrm{ dominant regular s.t. }\hat{w}^{-1}\xi\in \hat{\cal X}^*\}
\]
\begin{remark}
Since regular chamber closures are intersection of Weyl chambers closures in ${\cal X}^*_{\mathbb{R}}$, if $\hat{w}\xi$ is dominant for some $\hat{w}\in\hat{W}$ and $\xi\in\hat{\cal X}^*$ regular, then the same holds for each $\xi'$ lying the same regular chamber closure. Moreover, since the stabilizer in $W$ of the regular dominant chamber of $\hat{\cal X}^*$ (seen in ${\cal X}^*$) is exactly $W_P$, such a $\hat{w}\in\hat{W}$ is unique. Remark that for the case $P=B$, then $\hat{T}=T$ and $\hat{W}=W$.
\end{remark}
In this setting, we can extend Proposition \ref{unstablefullflags}. For $\hat{w}\in\hat{W}$, define
\[
W(\hat{w},\lambda,q):=\{w\in W,\exists\xi\in \hat{\cal X}^*\textrm{ with } \xi\textrm{ dominant regular },\langle w\lambda-\hat{w}q,\xi\rangle >0\}
\]
We get
\begin{theoreme}\label{unstableflags}
The subset of unstable points for $l=({\cal L}_{\lambda},q)$ is
\[
X^{us}({\cal L})=\bigcup\limits_{\hat{w}\in \hat{W}}\bigcup\limits_{w\in W(\hat{w},\lambda,q)}\hat{w}^{-1}X_w
\]
\end{theoreme}
\begin{proof}
The proof is the same as for Proposition \ref{unstablefullflags} after replacing ${\cal X}^*$ by $\hat{\cal X}^*$ and taking $\hat{w}$ in $\hat{W}$.
\end{proof}
This description makes a few questions naturally arise :\\
\textbullet \indent Let us denote by $S_d$ the set of $w\in W^P$ such that $wP/P$ is unstable relatively to some dominant regular one-parameter subgroup. Since $X^{us}$ is closed, if $w\in S_d$ and $w\geq v$, then $v\in S_d$. An interesting question would be to know what are the maximal (for the Bruhat order) elements of $S_d$, in order to describe the set of unstable points relatively to some dominant regular one-parameter subgroup (which will be denoted by $X^{us}_d$) as a minimal union of Schubert varieties. In particular, it would be interesting to know when it is a single Schubert variety.\\
\textbullet \indent An other interesting question is to know how the Schubert cells $\hat{w}X_{w_1}$ and $\hat{w}'X_{w_2}$ intersect for some $\hat{w},\hat{w}'\in \hat{W}$ and $w_1,w_2\in W^P$. One application would be to have a nice stratification of $X^{us}$ that would allow us to compute local cohomology groups with support in $X^{us}$ of invertible sheaves on $X$, and this would be helpful to compute cohomology groups of invertible sheaves on the quotient $Y(l)$.\\

\section{Case of a maximal parabolic subgroup}\label{sec4}

\subsection{Description of $S_d$}\label{sec4-1}

Let $G,P,B,T$ be as in the previous section, such that $P$ is a maximal parabolic subgroup of $G$, associated to the simple root $\alpha_i$, and let $X=G/P$. To be in the previous setting, we have to take $\hat{T}=\check{\varpi_i}(\mathbb{C}^*)$. We have $\textrm{Pic}(X)=\mathbb{Z}_{\varpi_i}$, and let $l=({\cal L}_{\lambda},q)$ be a $\hat{T}$-very ample invertible sheaf where $\lambda=n\varpi_i$ with $n>0$, and $q=m\varpi_i$. We also need the existence of a $w\in W^P$ sending $-\check{\varpi_i}$ to $\check{\varpi_i}$, which has to be the longest word $w_{0,P}$ (ie. $\hat{W}=\{1,w_{0,P}\}$) since $-\check{\varpi_i}$ lies in the closure of the antidominant Weyl chamber of ${\cal X}^*$, and since the stabilizer in $W$ of $\mathbb{Z}_{>0}\check{\varpi_i}$ is exactly $W_P$. Following the classification and notations in \cite{Bo}, this is the case only in the following types :
\begin{itemize}
\item Type $B_n$, $C_n$, $E_7$, $E_8$, $F_4$, $G_2$ ;
\item Type $A_n$ with odd $n$, and $i=\frac{n+1}{2}$ ;
\item Type $D_n$ with even $n$ ;
\item Type $D_n$ with odd $n$, and $i<n-1$ ;
\item Type $E_6$ with $i=2$ or 4
\end{itemize}
We will first give an answer to the first question with the following proposition :
\begin{theoreme}\label{instablemax}
When nonempty, $S_d$ has a unique maximal element $w$, and $X^{us}_d$ is the Schubert variety $\overline{X_w}$.
\end{theoreme}
Before starting to prove this proposition, remark that in this case, dominant regular weights can be expressed as $r\check{\varpi_i}$ with $r>0$. Let us denote $s_{\alpha_j}$ by $s_j$, and define for $k\geq 0$
\[
W^P(k):=\{w\in W^P, \exists\textrm{ a reduced expression }w=s_{i_p}\ldots s_{i_1}\textrm{ s.t. }s_i\textrm{ occurs at most }k\textrm{ times }\}
\]
We will state and prove the following lemma :
\begin{lemma}
\item[(1)] Let $w\in W^P$ and $\alpha_j\in\Delta$ such that $s_jw\in W^P$. Then 
\[
\langle w\lambda-q,\check{\varpi_i}\rangle\neq \langle s_{\alpha_j}w\lambda-q,\check{\varpi_i}\rangle\Leftrightarrow i=j
\]
\item[(2)] For all $w,w'\in W^P(k)\setminus W^P(k-1)$ and $k\geq 1$,
\[
\langle w\lambda,\check{\varpi_i}\rangle=\langle w'\lambda, \check{\varpi_i}\rangle
\]
\end{lemma}
\begin{proof}
For (1), take $\mu\in {\cal X}$ and write it as $\mu=\sum m_k\alpha_k$. Then $\langle \mu,\check{\varpi_i}\rangle$ is just the coefficient $m_i$. Hence the claim is equivalent to $(w\lambda, \alpha_i)\neq 0$, which is true since ${\cal L}_\lambda$ is very ample.\\
\indent (2) can be proven by induction on $k$, and by noticing thanks to (1) that if $w''\in W^P(k)\setminus W^P(k-1)$ such that $w''\leq w$, then the equality holds, hence we can reduce to the case of minimal $w, w'$, which have to be written respectively as $s_iu$ and $s_iv$ with $u,v\in W^P(k-1)\setminus W^P(k-2)$.
\end{proof}
We will also use the following lemma, which is due to V. Deodhar (cf. \cite{Ma} lemma 20, and \cite{LS} lemma 4.4 for a different formulation) :
\begin{lemma}
Let $Q$ be a parabolic subgroup of $G$, let $w,v\in W$ such that $v\leq w$ and $v\in W^Q$ (meant as a minimal length representative). Then the set
\[
C(w,v):=\{z\in W, v\leq z\leq w \textnormal{ and } zv^{-1}\in W_Q\}
\]
has a unique maximal element.
\end{lemma}
We will now prove Theorem \ref{instablemax}.
\begin{proof}
The point (2) of previous lemma implies that $S_d$ is actually a $W^P(k)$ for some $k$, hence it remains to show that it has a unique maximal element. Remark that maximal elements of $W^P(k)$ are actually in $W^P(k)\setminus W^P(k-1)$, and that if $w\in W^P$, then $s_jw\in W^P$ iff $\alpha_j\in w(\Phi^+\setminus \Phi_P^+)$. 
Let $k>0$, and let $w$ be a maximal element of $W^P(k)$. We will assume the following claim to be true : $w$ can be written as $w=us_iv$ with $u\in W_P$ and $v$ maximal in $W^P(k-1)$. Then Deodhar's lemma gives us by induction the unicity of $w$ : if we denote by $v$ the unique maximum of $W^P(k-1)$, then $w$ is the maximum of $C(w_0,s_iv)$.\\
\indent
It remains now to prove the claim. Let $w\in W^P(k)$ be maximal, and write it as $w=us_iv$ for some $v\in W^P(k-1)$ and $u\in W_P$. Let $v'\in W^P(k-1)$ be maximal such that $v\leq v'$. Then we have $s_iv\leq s_iv'$. Deodhar's lemma for $C(w_{0,P},s_iv')$ gives us a unique $u'\in W_P$ such that $z=u's_iv'$ is maximal in $W^P(k)$, and this expression is reduced. Applying Deodhar's lemma again for $C(w_{0,P},s_iv)$ shows that $z=w$, which shows the claim.
\end{proof}
\begin{remark} If we denote by $X^{us}_{ad}$ the set of unstable points relatively to some antidominant regular one-parameter subgroup, then the theorem implies that $X^{us}_{ad}=w_{0,P}\overline{X_w}$ for some $w\in W^P$. Moreover, since the intersection of regular chambers closures is reduced to 0, $X^{us}_d$ and $X^{us}_{ad}$ are disjointed. Set 
\[
M^P:=\{w\in W^P, \exists k \textrm{ s.t. }w \textrm{ maximal in }W^P(k)\}
\]
and denote its elements by $w_k$ where $k$ corresponds to the $W^P(k)$ in which they are maximal. Let $\overline{X_{w_{-1}}}=\emptyset$ and let $k_{\max}$ be the unique integer such that $w_{0,P}=w_{k_{\max}}$. Then we obtain the following corollary :
\end{remark}
\begin{corollary}
There exists $-1\leq k,k'\leq k_{\max}$ such that $X^{us}(l)=\overline{X_{w_k}}\sqcup w_{0,P}\overline{X_{w_{k'}}}$.
\end{corollary}
\begin{remark}
The maximal element of $W^P(k)$ can be explicited thanks to the following product $\star$ defined for $w\in W$ and $\alpha_j\in\Delta$ by
\[
s_j\star w:=\left\lbrace\begin{array}{rr}s_iw&\textrm{ if }l(s_iw)>l(w)\\w&\textrm{ else}\end{array}\right.
\]
This product naturally appears as the limit of product in the Hecke algebra of $W$ when the parameter $q\to 0$, and this product is associative. This product also occurs when looking at the image product of preimages of Schubert varieties : for $w, w'\in W$, the image of the map
\[
\overline{BwB}\times^{B}\overline{Bw'B}/B\to G/B
\]
is projective, $B$-invariant and irreducible, hence it is a Schubert variety, which is $\overline{X_{w\star w'}}$. More detail about this $\star$-product can be found in \cite{BM}. Then if we denote by $z$ the longest word in $W_P$, the maximal element of $W^P(k)$ is given by $(z\star s_i)^{\star k}$. We should also remark that the unicity of the maximum of $W^P(k)\setminus W^P(k-1)$ implies he unicity of the minimum of $W^P(k_{\max}-k+1)\setminus W^P(k_{\max}-k)$, which will be denoted by $v_{k_{\max}-k+1}$.
\end{remark}
Thanks to the previous result, we know that the possible dimension of $X^{us}$ lies in the set $D_P:=\{l(w),w\in M^P\}$, which uniquely depends on the root system of $G$.\\
\indent When the quotient is not to bad, knowing the codimension of $X^{us}(l)$ implies the vanishing of some cohomology groups of invertible sheaves of the quotient. Here, not to bad means that $X$ is Cohen-Macaulay, that the codimension of $X^{us}(l)$ is at least two, and that this quotient is geometric (ie. $X^{ss}(l)=X^{s}(l)$). Let $l\in C^{\hat{T}}(X)$, let $\pi:X^{ss}(l)\to Y$ be the quotient, and let $j:X^{ss}(l)\to X$ the inclusion. In that case, invertible sheaves ${\cal M}$ on $Y$ can be lifted to a $\hat{T}$-linearized invertible sheaf $\tilde{\cal M}$ on $X$ such that ${\cal M}=(\pi_*j^*\tilde{\cal M})^{\hat{T}}$. Note that this makes sense, since open subsets of $Y$ lift to $\hat{T}$-invariant open subsets of $X^{ss}(l)$. The condition of $X$ being Cohen-Macaulay allows us to say that for any locally free sheaf ${\cal F}$ on $X$, $H^i(X,{\cal F})=H^i(X^{ss}(l),{\cal F})$ for $i<\textrm{codim}(X^{us}(l))-1$ thanks to the localization long exact sequence. More details about these generalities can be found in \cite{De} 4.1 and 4.4.\\
\indent Since in our case $X$ is a flag variety, cohomology groups of its invertible sheaves are fully understood thanks to Borel-Weil-Bott theorem (cf. \cite{CS} for general $G/P$). Since $H^i(Y,{\cal M})=H^i(X^{ss}(l),\tilde{\cal M})^{\hat{T}}$ for ${\cal M}\in\textrm{Pic}(Y)$, we get the following :
\begin{proposition}\label{vanish}
Let $l\in C^{\hat{T}}(X)$ such that the quotient $X^{ss}(l)\to Y$ is geometric. Then for any invertible sheaf ${\cal M}\in \textnormal{Pic}(Y)$, $H^i(Y,{\cal M})=0$ for $0<i<\textnormal{codim}(X^{us}(l))-1$ and if $Y$ is Gorenstein, for $\dim(X)>i>\dim(X)-\textnormal{codim}(X^{us}(l))+1$.
\end{proposition}
\begin{proof}
Since $P$ is a maximal parabolic subgroup of $G$, for any invertible sheaf $\tilde{\cal M}$ on $X$, $H^*(X,\tilde{\cal M})\neq 0$ implies $i=0$ or $\dim(X)$. To be more precise, $H^*(X,\tilde{\cal M})$ is concentrated in degree 0 if $\tilde{\cal M}$ is in the closure $\overline{A}$ of the ample cone, concentrated in top degree if $\omega_X\otimes \tilde{\cal M}^{\otimes -1}$ lies in $\overline{A}$, and 0 in the other cases. The vanishing in the lase case can be shown by using Kodaira vanishing theorem saying that $H^i(X,{\cal F}\otimes \omega_X)=0$ for $i>0$ and ${\cal F}$ ample (or equivalently, that $H^i(X,{\cal F}^{\otimes -1})=0$ for $i<\dim(X)$), saying us that $H^*(X,\tilde{\cal M})$ is concentrated in top degree. But using Serre duality, it is also concentrated in degree 0, hence it is 0.\\
\indent The claim for $0<i<\textrm{codim}(X^{us}(l))-1$ follows from the previous discussion, and it is given by Serre duality for $\dim(X)>i>\dim(X)-\textrm{codim}(X^{us}(l))+1$. Note that when $\textrm{codim}(X^{us}(l))<2$, this statement is empty.
\end{proof}
\begin{remark}
We already know from \cite{HR} that $Y$ is Cohen-Macaulay, but it is not necessarly factorial. Following the criterion for factoriality given in \cite{Dr} 8.1.2, we get that if there exists an open set $X_0\subset X^{s}$ on which $\hat{T}$ acts trivially such that $X^{ss}\setminus X_0$ is of codimension at least two, then $Y$ is Gorenstein iff every $\hat{T}$-linearized invertible sheaf on $X^{ss}$ descends to $Y$ (ie. it is the restriction to $X^{ss}$ of a $\tilde{\cal M}$ for some ${\cal M}\in \textrm{Pic}(Y)$).
\end{remark}
Saying that the quotient $X^{ss}\to Y$ is geometric is equivalent to say that all strict inequalities defining $X^{us}(l)$ can be taken as large inequalities. Hence this condition can be restated as :
\[
\forall w\in M^P, \langle w\lambda -q,\check{\varpi_i}\rangle\neq 0
\]
\begin{example}
The following examples are the one that led me to look at this problem of unstable points for torus actions on flag varieties, namely the quotients $X//\mathbb{C}^*$ in one of the following cases :
\begin{itemize}
\item[1)] $G=\textrm{SL}_6$, $\alpha_i=\alpha_3$ ;
\item[2)] $G=\textrm{Sp}_6$, $\alpha_i=\alpha_3$ ;
\item[3)] $G=\textrm{SO}_{12}$, $\alpha_i=\alpha_6$ ;
\item[4)] $G$ of type $E_7$, $\alpha_i=\alpha_7$ ;
\end{itemize}
all these cases being considered with trivial linearizations. It is shown in \cite{De} 3.2 that these quotients are exactly the wonderful compactifications of symmetric spaces whose restricted root system is of type $A_2$. Then the reduced expressions for $w_{0,P}$ such that maximal suffixes of this expression having $k$ occurrences of $s_i$ is exactly $w_k$ are respectively :
\begin{itemize}
\item[1)] $s_{342312543}$ ;
\item[2)] $s_{323123}$ ;
\item[3)] $s_{645346234512346}$ ;
\item[4)] $s_{765432456713452645341234567}$.
\end{itemize}
Thus the codimensions of unstable $X^{us}$ are respectively 4, 3, 6 and 10. This gives for free the vanishing of cohomology groups $H^i(Y,{\cal M})$ for invertible sheaves on ${\cal M}$ for $i$ in the range given by Proposition \ref{vanish}. This was already proven in \cite{Tc1} and \cite{Tc2} but this was done using local cohomology and Cousin complexes machinery.
\end{example}

\subsection{Variation of quotients}\label{sec4-2}
We will say a few words about variation of GIT quotients in our case. Since $P$ is a maximal parabolic, its Picard group is reduced to $\mathbb{Z}$, and once we fix a very ample invertible sheaf ${\cal L}_{\lambda}$, a ray in $\textrm{NS}^{\hat{T}}(X)_{\mathbb{Q}}$ generated by a $\hat{T}$-linearized ample invertible sheaf can be identified with a fractional character $q\in\hat{\cal X}_{\mathbb{Q}}\simeq \mathbb{Q}$ having semi-stable points. By Theorem \ref{unstableflags}, having semi-stable points is equivalent to
\[
\exists w,w'\in W^P, \langle w\lambda-q,\check{\varpi_i}\rangle \leq 0 \textrm{ and } \langle w'\lambda+q,\check{\varpi_i}\rangle \leq 0
\]
which is equivalent to the same condition with $w=w'=w_{0,P}$, ie.
\[
\langle -\lambda,\check{\varpi_i}\rangle \leq \langle q,\varpi_i\rangle \leq \langle \lambda,\check{\varpi_i}\rangle
\]
Let $q_{\max}:=\langle \lambda,\check{\varpi_i}\rangle$. The partition of $C^{\hat{T}}(X)$ in wall and chambers will be equivalent to its induced partition of $[-q_{\max},q_{\max}]$. Remark that here, walls and cells coincide, since they both correspond to points in $[-q_{\max},q_{\max}]$. The walls will be given by $q$ for which the quotient is not geometric, ie. $\exists w\in W^P$ such that
\begin{equation}\label{doublecancel}
\langle w\lambda-q,\check{\varpi_i}\rangle = 0 \textrm{ or } \langle w\lambda+q,\check{\varpi_i}\rangle = 0
\end{equation}
but the first condition being satisfied for $w$ implies the second condition being satisfied for $ww_{0,P}$. Hence for $q$ in a wall, there are both semi-stable points respectively to dominant one-parameter subgroups, and respectively to antidominant ones, which are not stable. We now can show the following :
\begin{proposition}\label{uspointsmaxpara}
There are exactly $k_{\max}$ chambers having semi-stable points, and they correspond to the $C(k):=]\langle w_{k+1}\lambda,\check{\varpi_i}\rangle , \langle w_{k}\lambda,\check{\varpi_i}\rangle [$ for $0\leq k \leq k_{\max}-1$, and for $q\in C(k)$, $X^{us}(q)=\overline{X_{w_k}}\sqcup w_{0,P}\overline{X_{w_{k_{\max}-1-k}}}$
\end{proposition}
\begin{proof}
The description of chambers with semi-stable points is obvious since the walls are exactly the $\{\langle w_{k}\lambda,\check{\varpi_i}\rangle\}$. The equation \ref{doublecancel} implies that if $X^{us}(q)=\overline{X_{w_k}}\sqcup w_{0,P}\overline{X_{w_{k'}}}$ for $q$ in a chamber, and if $q'$ lies in a different adjacent chamber, then $X^{us}_{d}(q)\neq X^{us}_{d}(q')$ and $X^{us}_{ad}(q)\neq X^{us}_{ad}(q')$. Thus we get
\[
X^{us}(q')=\overline{X_{w_{k+1}}}\sqcup w_{0,P}\overline{X_{w_{k'-1}}} \textrm{ or } X^{us}(q')=\overline{X_{w_{k-1}}}\sqcup w_{0,P}\overline{X_{w_{k'+1}}}
\]
the first case being equivalent to $q-q'>0$. Since for $q>>0$ we have $X^{us}_{d}(q)=\emptyset$ and $X^{us}_{ad}(q)=X$, the description of $X^{us}(q)$ for $q\in C(k)$ follows.
\end{proof}
\begin{remark}
The maps $f_{F,F'}$ from the Proposition \ref{inverselimit} give the collection of nonempty GIT quotients a structure of inverse system, and one can consider its inverse limit. Having such a nice description of unstable points of GIT quotients of $X$ by $\check{\varpi_i}$, an interesting question would be to have a nice description of this limit $\underline{X//\check{\varpi_i}}$. Such limits can be interesting to look at : for example, let $V$ be a $n$-dimensional $\mathbb{C}$-vector space. M. Thaddeus showed in \cite{Th2} that one can in such a way obtain the moduli spaces of complete collineations from $V$ to $V$ of maximal ranks, of complete quadrics in $\mathbb{P}(V)$, or of complete skew forms on $V$.
\end{remark}

\section*{Acknowledgements}
I would like to thank O. Mathieu and M. Pelletier for useful discussions.


\bibliography{mybibfile}

\begin{thebibliography}{10}

\bibitem{Bo}
{\sc N.~Bourbaki}, {\em Groupes et Alg\`ebres de Lie, Chapitres IV-VI},
  Actualit\'es Scientifiques et Industrielles, No. 1337, Hermann, Paris, 1968.

\bibitem{BP}
{\sc M.~Brion and C.~Procesi}, {\em Action d'un tore dans une vari\'et\'e
  projective}, vol.~92 of Progr. Math., Birkh\"auser Boston, Boston, MA, 1990.

\bibitem{BM}
{\sc A.~S. Buch and L.~C. Mihalcea}, {\em Curve neighborhoods of {S}chubert
  varieties}, J. Differential Geom., 99 (2015), pp.~255--283.

\bibitem{CS}
{\sc A.~{\v C}ap and J.~Slov\'ak}, {\em Parabolic geometries. {I}}, vol.~154 of
  Mathematical Surveys and Monographs, American Mathematical Society,
  Providence, RI, 2009.
\newblock Background and general theory.

\bibitem{De}
{\sc B.~Dejoncheee}, {\em On differential operators on complete symmetric
  varieties of type ${A}_1$ and ${A}_2$}, arXiv:1609.06998.

\bibitem{DH}
{\sc I.~V. Dolgachev and Y.~Hu}, {\em Variation of geometric invariant theory
  quotients}, Inst. Hautes \'Etudes Sci. Publ. Math.,  (1998), pp.~5--56.
\newblock With an appendix by Nicolas Ressayre.

\bibitem{Dr}
{\sc J.-M. Dr\'ezet}, {\em Luna's slice theorem and applications}, in Algebraic
  group actions and quotients, Hindawi Publ. Corp., Cairo, 2004, pp.~39--89.

\bibitem{HR}
{\sc M.~Hochster and J.~L. Roberts}, {\em Rings of invariants of reductive
  groups acting on regular rings are {C}ohen-{M}acaulay}, Advances in Math., 13
  (1974), pp.~115--175.

\bibitem{LS}
{\sc V.~Lakshmibai and C.~S. Seshadri}, {\em Geometry of {$G/P$}. {V}}, J.
  Algebra, 100 (1986), pp.~462--557.

\bibitem{La}
{\sc R.~Lazarsfeld}, {\em Possitivity in algebraic geometry {I}: Classical
  setting: line bundles and linear series.}, vol.~48 of Ergebnisse der
  Mathematik und ihrer Grenzgebiete, Springer-Verlag, Berlin, 2004.

\bibitem{Ma}
{\sc O.~Mathieu}, {\em Filtrations of {$B$}-modules}, Duke Math. J., 59 (1989),
  pp.~421--442.

\bibitem{MFK}
{\sc D.~Mumford, J.~Fogarty, and F.~Kirwan}, {\em Geometric invariant theory},
  vol.~34 of Ergebnisse der Mathematik und ihrer Grenzgebiete (2) [Results in
  Mathematics and Related Areas (2)], Springer-Verlag, Berlin, third~ed., 1994.

\bibitem{ST}
{\sc H.~Seppänen and V.~Tsanov}, {\em Unstable loci in flag varieties and
  variation of quotients}, arXiv:1607.04231.

\bibitem{Tc1}
{\sc A.~Tchoudjem}, {\em Cohomologie des fibr\'es en droites sur les
  vari\'et\'es magnifiques de rang minimal}, Bull. Soc. Math. France, 135
  (2007), pp.~171--214.

\bibitem{Tc2}
\leavevmode\vrule height 2pt depth -1.6pt width 23pt, {\em Sur la cohomologie
  \`a support des fibr\'es en droites sur les vari\'et\'es sym\'etriques
  compl\`etes}, Transform. Groups, 15 (2010), pp.~655--700.

\bibitem{Th2}
{\sc M.~Thaddeus}, {\em Complete collineations revisited}, Math. Ann., 315
  (1999), pp.~469--495.

\end{thebibliography}

\end{document}